\documentclass[11pt]{amsart}
\usepackage{graphics}

\usepackage{dsfont}
\makeatletter
\newcommand*\bigcdot{\mathpalette\bigcdot@{.5}}
\newcommand*\bigcdot@[2]{\mathbin{\vcenter{\hbox{\scalebox{#2}{$\m@th#1\bullet$}}}}}
\makeatother

\usepackage{amsmath,amssymb}
\usepackage{amsmath}
\usepackage{amsthm}
\usepackage{amsfonts}
\usepackage{xcolor}
\usepackage[english]{babel}
\usepackage[margin=1.09in]{geometry}
\usepackage[colorlinks=true,
linkcolor=blue]{hyperref}

\definecolor{pink}{rgb}{1,0,1}

\usepackage[latin9]{inputenc}
\usepackage{amsmath}
\usepackage{amsfonts}
\usepackage{amssymb}
\usepackage{amsthm}

\newtheorem{theo}{Theorem}[section]
\newtheorem{prop}[theo]{Proposition}

\newtheorem{lemm}[theo]{Lemma}

\theoremstyle{definition}

\newtheorem{rema}[theo]{Remark}

\newcommand{\nwc}{\newcommand}
\nwc{\eps}{\epsilon}
\nwc{\vareps}{\varepsilon}
\nwc{\Oph}{\operatorname{Op}_\hbar}
\nwc{\la}{\langle}
\nwc{\ra}{\rangle}

\nwc{\mf}{\mathbf} 
\nwc{\blds}{\boldsymbol} 
\nwc{\ml}{\mathcal} 

\nwc{\defeq}{\stackrel{\rm{def}}{=}}

\nwc{\cE}{\ml{E}}
\nwc{\cN}{\ml{N}}
\nwc{\cO}{\ml{O}}
\nwc{\cP}{\ml{P}}
\nwc{\cU}{\ml{U}}
\nwc{\cV}{\ml{V}}
\nwc{\cW}{\ml{W}}
\nwc{\tU}{\widetilde{U}}
\nwc{\IN}{\mathbb{N}}
\nwc{\IR}{\mathbb{R}}
\nwc{\IZ}{\mathbb{Z}}
\nwc{\IC}{\mathbb{C}}
\nwc{\IT}{\mathbb{T}}
\nwc{\tP}{\widetilde{P}}
\nwc{\tPi}{\widetilde{\Pi}}
\nwc{\tV}{\widetilde{V}}
\nwc{\supp}{\operatorname{supp}}
\nwc{\rest}{\restriction}

\newcommand{\N}{{\mathbb N}}
\newcommand{\R}{{\mathbb R}}

\renewcommand{\d}{\partial}

\renewcommand{\phi}{\varphi}

\newcommand{\M}{\mathcal{M}}

\newcommand{\ep}{\varepsilon}

\title [Upper bounds on the size of nodal sets] {Upper bounds on the size of nodal sets for Gevrey and quasianalytic Riemannian manifolds}

\author[Hezari]{Hamid Hezari }
\address{Department of Mathematics, UC Irvine, Irvine, CA 92617, USA} \email{hezari@math.uci.edu}

\date{\today}

\begin{document}


\maketitle

\begin{abstract}
We find new polynomial upper bounds for the size of nodal sets of eigenfunctions when the Riemannian manifold has a Gevrey or quasianalytic regularity. 
 
\end{abstract}

\section{Introduction} Let $(M,g)$ be a smooth compact connected boundaryless Riemannian manifold of dimension $n$.  Suppose $-\Delta_g$ is the positive Laplace-Beltrami operator on $(M, g)$ and $\psi_\lambda$ is an eigenfunction of $-\Delta_g$ with eigenvalue $\lambda$. We are interested in upper bounds on the $(n-1)$-dimensional Hausdorff measure  $\mathcal{H}^{n-1} (Z_{\psi_\lambda})$ of the nodal set $Z_{\psi_\lambda}=\{ x \in M; \; \psi_\lambda(x)=0\}$ of $\psi_\lambda$. When $(M, g)$ belongs to the analytic class $C^\omega$, Donnelly and Fefferman \cite{DF1} proved the upper bound $C \sqrt{\lambda}$. In the $C^\infty$ case and for $n \geq 3$, a  result of Logunov \cite{La} gives a polynomial upper bound of the form $C\lambda^ \alpha$ for some non-explicit but apparently large $\alpha > \frac{1}{2}$ depending only on $n$. For $n=2$, Logunov and Malinnikova \cite{LMb} find upper bounds of the form $C\lambda^{\frac{3}{4}- \beta}$ for some small universal $\beta \in (0, \frac{1}{4})$. In this paper, we are interested in the intermediate cases of regularity, namely ultradifferentiable Riemannian manifolds, also called Denjoy-Carleman  (or Roumieu) classes $C^{\mathcal M}$ that lie between the analytic and smooth classes, i.e. $C^\omega \subset C^{\mathcal M} \subset C^\infty$. We are mainly interested in Gevrey classes and quasianalytic classes of Riemannian manifolds and obtain new upper bounds for the size of nodal sets for these cases. 

Our first result in this direction is:

\begin{theo} \label{main} Let $(M, g)$ be a compact connected boundaryless Riemannian manifold that belongs to the Gevrey class $G^s$, $s \geq 1$. Let $\psi_\lambda$ be an eigenfunction of $-\Delta_g$ with eigenvalue $\lambda \geq 0$.  
	Then 
	$$\mathcal{H}^{n-1} (Z_{\psi_\lambda}) \leq C {\lambda}^{\frac{s}{2}} ,$$
where the constant $C= C(M, g, s)$ depends only on $(M, g)$ and $s$. 
\end{theo}
 For $s>1$ sufficiently small, Theorem \ref{main} improves the polynomial upper bounds of \cite{La} and \cite{LMb}, of course only in the Gevrey case.  The proof uses Crofton's formula, hypoelliptic estimates in the Gevrey setting, and also a fascinating and powerful result of \cite{LMa} on the measure of sublevel sets of eigenfunctions. Our method also applies to the more general (Denjoy-Carleman) classes of Riemannian manifolds. In particular, when $(M, g)$ belongs to the quasianalytic class $C^{\mathcal M}$, $\mathcal M_k= \left ( \log (k +e) \right)^{k}$, we prove: 
 \begin{theo}\label{quasi}
$\mathcal{H}^{n-1} (Z_{\psi_\lambda}) \leq C \sqrt{\lambda} \log (\lambda +e)$.  
   \end{theo}
This quasianalytic case is interesting because up to a factor $\log (\lambda +e)$, the upper bound agrees with Yau's upper bound conjecture \cite{Y82}.  

More generally, if $(M, g) \in C^{\mathcal M}$, where $\mathcal M$ is a \textit{regular sequence} with \emph {moderate growth} (see the next section for definitions), then we have:

\begin{theo}\label{general}
	$\mathcal{H}^{n-1} (Z_{\psi_\lambda}) \leq C \sqrt{\lambda }\; \left ( \mathcal M _{ \lfloor \sqrt{\lambda} \rfloor } \right ) ^{\frac{1}{\lfloor \sqrt{\lambda} \rfloor }}$.  
\end{theo}
When $\mathcal M_k = k!^{s-1}$, $s \geq 1$, we have $C^{\mathcal M} = G^s$ and thus Theorem \ref{main} is in fact just a particular case of Theorem \ref{general}. We can generate more results by considering more examples of weight sequences such as $M_k = \log^{sk}(k+e)$, $s \geq 0$, which would provide the upper bounds $$\sqrt{\lambda} \log ^{s} (\lambda + e).$$
Although Theorem \ref{general} is more general, but to keep the presentation smooth, we focus mainly on the Gevrey case and only highlight the changes needed in the notationally bothersome and less known case of Denjoy-Carleman classes. 

It is worth mentioning that one can probably use the \textit{Bang degree} (see \cite{NSV}), combined with the method of this paper, to obtain the above estimates. It would also be interesting to study the case of manifolds with boundary (see \cite{LMNN}) and  Steklov eigenfunctions (see \cite{Ze2}) in the Gevrey setting. 

\section{Hypoellipticity of $\Delta_g$ in the ultradifferentiable setting}
We start this section by reviewing the ultradifferentiable classes that we are interested in.
\subsection{The Gevrey class and Denjoy-Carleman classes}
Let $U \subset \R^n$ be an open set and let $s \geq 1$. The Gevrey class $G^s(U)$ is defined as the space of smooth functions $f$ on $U$ such that for each compact $K \subset U$ there exist $C$ and $A$ positive such that for all multiindices $\alpha \in \mathbb N^n$, one has
$$ \sup_K |\d^\alpha f(x)| \leq C A^{|\alpha|} |\alpha|!^s,$$
where $|\alpha| = \alpha_1 + \dots + \alpha_n$.
We refer to such functions as $s$-Gevrey functions. 
When $s=1$ we have $G^1(U) = C^\omega(U)$, i.e. the class of analytic functions on $U$. For $s \leq s'$ we have $G^{s} \subset G^{s'}$. The space $G^s(U)$  is closed under addition and multiplication. 

A map $f=(f_1, \dots, f_n): U\subset \R^n \to \R^n $ is called $s$-Gevrey if each $f_j$ is $s$-Gevrey. The composition of two $s$-Gevrey maps is $s$-Gevrey if the composition is well-defined.   

A smooth manifold $M$ belongs to the Gevrey class $G^s$ if it has an atlas $\{(U_\beta, \phi_\beta); \beta \in \mathcal A\}$ whose transition maps $\phi_{\beta_1} \circ \phi_{\beta_2}^{-1}$  are in $G^s$. An $s$-Gevrey structure on $M$ is a maximal atlas of $s$-Gevrey charts. We say a Riemannian manifold $(M, g)$ belongs to $G^s$ if in addition to $M$, the metric $g$ also belongs to the Gevrey class of degree $s$, which is equivalent to $g_{ij} \in G^s$ on every Gevrey chart $(U, \phi)$ belonging to the Gevrey structure.  As examples of $G^s$ Riemannian manifolds, consider any real analytic boundaryless manifold $M$ (such as a sphere or a torus), then there is an infinite dimensional space of $s$-Gevrey metrics $g$ that $M$ can be equipped with. This can be done first by defining Gevrey metrics locally on Gevrey charts, then use an appropriate partition of unity to obtain a global Gevrey metric. In the construction of the partition of unity, one needs to use the bump functions,
\begin{equation} 
g_\gamma(x)= \begin{cases}  e^{- \frac{1}{(x(1-x))^\gamma}}  & x \in (0,1 ), \\ 0  &  x \notin (0, 1). \end{cases}
\end{equation}
For each $\gamma >0$, the function $g_\gamma$ belongs to $G^{1+\frac{1}{\gamma}}$. The standard partition of unity theorem uses $g_1$ which is $2$-Gevrey and is not suitable for localizing functions within the class $G^s$ if $s<2$.

In fact there are more general classes of smooth functions associated to increasing sequences $\M=\{ \M_k \}_{k \in \N}$ of positive real numbers, called Denjoy-Carleman or Roumieu functions, and denoted by $C^\M$. Given $U \subset \R^n$ open, we define 
$C^\M(U)$ to be the class of functions $f$ on $U$ such that for all $K \subset U$ compact, there exist $C$ and $A$ such that for all $\alpha \in \N^n$,
$$ \sup_K |\d^\alpha f(x)| \leq C A^{|\alpha|} |\alpha|! \; \M_{|\alpha|}.$$
The sequence $\M=\{ \M_k \}_{k \in \N}$  is called the weight sequence.  By this definition, the $s$-Gevrey class corresponds to the weight sequence $\M_k = k!^{s-1}$. Following \cite{F}, we say $\M$ is a \emph{regular wight sequence} if 
\begin{align*} (i) \; \; & \M_0 =1,\\
                            (ii) \;  \; &\M_k \; \text{is increasing},\\
                            (iii)\; \; & \M_k^2 \leq \M_{k-1}\M_{k+1}, \quad \forall k \in \N \quad \text{(logarithmic convexity)}\\
                            (iv)\; \;  & \sup_{k \in \N^{>0}} \left( \frac{\M_{k+1}}{\M_k} \right )^{1/k} < \infty.
\end{align*}
In addition, a sequence $\M$ has moderate growth if 
\begin{align}\label{moderate}  (v)\; \; &  \sup_{k, m \in \N^{>0}} \left( \frac{\M_{k+m}}{\M_k \M_m} \right )^{1/(k+m)} < \infty. \qquad \qquad \qquad \qquad \qquad\end{align}
Conditions $(i)-(iii)$  imply that the class $C^\M$ is closed under addition, multiplication, and composition. Condition (iv) guarantees that the class is closed under derivation (see for example the paper \cite{KMR} for a broad review). The moderate growth condition $(v)$, in addition to regularity, would provide $C^\M$-hypoellipticity for elliptic partial differential operators with coefficients in $C^\M$. Note that condition $(v)$ implies $(iv)$ and it also forces $C^\M \subset G^s$ for some $s$ (see for example Lemma 3.2.1 of \cite{Th}). By a result of Denjoy and Carleman, the class $C^\M$ with $\M$ satisfying $(i)-(iii)$, is quasianalytic, meaning, functions which are flat at a point are identically zero, if and only if
$$ \sum_{k=0}^\infty \frac{\M_k}{(k+1)\M_{k+1}} =\infty.$$
For example the weight sequence $M_k= \log^{sk}(k+e)$ produces a quasianalytic class if $0<s\leq 1$, and a non-quasianalytic class if $s>1$. A non-quasianalytic class always accepts a partition of unity. 
Similar to the notion of Gevrey Riemannian manifolds, we can define Denjoy-Carleman Riemannian manifolds (see \cite{F}). Examples of compact quasianalytic Riemannian manifolds can not be provided via a partition of unity. Instead, one can use the implicit function theorem for $C^\M$-mappings (see \cite{Ko79}), or for example take a compact analytic manifold that is embedded in the euclidean space, and for the metric consider a conformal factor that is in a desired quasianalytic class.  

\subsection{Gevrey hypoellipticity of Laplacian and Gevrey regularity of eigenfunctions}

A partial differential operator $P$ with smooth coefficients on an open set $U \subset \R^n$  is called hypoelliptic if $Pu \in C^\infty(U)$ implies that $u \in C^\infty(U)$. It is known that elliptic partial differential operators (in fact more generally, elliptic pseudodifferential operators) are hypoelliptic. Similarly, there is a notion of Gevrey hypoellipticity, requiring that whenever $Pu \in G^s$ then $u \in G^s$. When $s=1$, this property is the same as analytic hypoellipticity.  Boutet de Monvel and Kre\'e \cite{BK} proved that any elliptic partial differential operator with $s$-Gevrey coefficients is $G^s$ hypoellliptic (the case $s=1$ was already known). Therefore, an eigenfunction of the Laplacian of a Gevrey Riemannian manifold must be Gevrey. However, in this paper we need the following quantitative estimates. 
\begin{prop} \label{Hypo}
	Let $(M, g)$ ba compact connected boundaryless Riemannian manifold. Suppose $(M, g)$ belongs to the class $G^s$. Then every eigenfunction $\psi_\lambda$ of $\Delta_g$ also belongs to the class $G^s$. Moreover, for every $s$-Gevrey chart $(U, \phi)$ and every compact $K \subset U$, there exist $C$ and $A$ independent of $\lambda$ such that
	\begin{equation} \label{Hypo-good} \sup_{K}|\d^\alpha \psi_\lambda| \leq C e^{\sqrt{\lambda}} A^{|\alpha|} \alpha!^s \sup_U{| \psi_\lambda}|. \end{equation}
\end{prop}
\begin{rema}
	The factor $e^{\sqrt{\lambda}}$ can be replaced with any $e^{\ep\sqrt{\lambda}}$ at the cost of allowing the constants $C$ and $A$ to be dependent on $\ep$ (certainly $A \geq c/\ep$).  
\end{rema}
\begin{rema} One could follow the argument of \cite{DF1} (see proof of Lemma 7.1) and rescale balls of radius $\lambda^{-1/2}$ to balls of radius one and then apply a  hypoellipticity  argument as in \cite[Page 176]{Ho} to obtain estimates of the form:
\begin{equation} \label{Hypo-bad} \sup_{K}|\d^\alpha \psi_\lambda| \leq C (\sqrt{\lambda}A)^{|\alpha|} \alpha!^s \sup_U{| \psi_\lambda}|. \end{equation}
Although estimates \eqref{Hypo-bad} are better than \eqref{Hypo-good} for small values of $|\alpha|$, but they are much worse when $|\alpha| \geq \sqrt{\lambda}$. For our purposes estimates \eqref{Hypo-good} play a key role. 
\end{rema}

\begin{proof}[\textbf{Proof of Proposition \ref{Hypo}}]
	We consider the function, 
	\begin{equation} \label{h} h(x, t) = e^{\sqrt{\lambda}t} \psi_\lambda(x), \end{equation}
	on the product manifold $\widetilde M= M \times \R$. Note that if we denote $\Delta_{\tilde g}= \d_t^2 + \Delta_g$ to be the Laplace-Beltrami operator associated to the metric $\tilde g$, defined as the product of $g$ and the Euclidean metric $e_0$ on $\R$, then $\Delta_{\tilde g} h =0$. In other words, $h$ is a harmonic extension of $\psi_\lambda(x)$ into $\widetilde M= M \times \R$. 
	
	Clearly, since $g \in G^s$, we have that $\tilde g \in G^s$ and $\Delta_{\tilde g}$ is an elliptic differential operator with coefficients in $G^s$. By Corollary 2.14 of \cite{BK}, we have $h \in G^s$, which implies that $\psi_\lambda \in G^s$.  To obtain the estimates \eqref{Hypo-good}, we shall use another result of \cite[Prop 2.13]{BK} which implies that if $P$ is an elliptic differential operator of order $2$ in $\Omega$ with coefficients in $G^s$, then there exists an elliptic pseudodifferential operator $Q$  of order $-2$ such that $R=I - QP$ is a smoothing operator whose distributional kernel $R(x, y)$ belong to the Gevrey class $G^s(\Omega  \times \Omega)$.  Now let $(U, \phi)$ be an arbitrary Gevrey chart of $M$ and consider $P = \Delta_{\tilde g}$ on $ \tilde U = U \times (-2, 2)$. Since $Ph=\Delta_{\tilde g} h=0$, we get $R h = h$. For an arbitrary compact subset $K \subset U$, we define the compact subset $\tilde K = K \times [-1, 1]$ of $\tilde U$. Then from $h =Rh = \int R(x, y) h(y) dy$ and the properties of $R$, we obtain
	$$\sup_{\tilde K}|\d^\alpha h| \leq C  A^{|\alpha|} \alpha!^s \sup_{\tilde U}{|h|}.$$
	Estimate \eqref{Hypo-good} follows from this and the definition \eqref{h} of $h$ in terms of $\psi_\lambda$.

	\end{proof}

\subsection{Denjoy-Carleman hypoellipticity} As mentioned earlier, if we add the moderate growth assumption $(v)$ to a regular sequence $\M$, then any elliptic partial differential operators with coefficients in $C^\M$ is $C^\M$-hypoelliptic. This result follows from Theorem 4.1 of Albanese, Jornet, and Oliaro \cite{Albanese} (see also Theorem 6.1 of \cite{F} for systems), which states that if $P$ is a linear differential operator (not necessarily elliptic) on $U \subset \R^n$ with coefficients in $C^\M(U)$, then 
$$\text{WF}_\M(u) \subset \text{WF}_\M(Pu) \cup \Sigma, $$
where $\Sigma$ is the characteristic set of $P$ and $\text{WF}_\M$ denotes the wavefront set associated to the class $C^\M$.\footnote{A result of this type was first proved by H\"ormander for operators with analytic coefficients.}  If $Pu\in C^\M$ and $P$ is elliptic, then we obtain that $\text{WF}_\M(u) = \emptyset$, thus $u \in C^\M(U)$. In particular, this implies that an eigenfunction of a $C^\M$-Riemannian manifold $(M, g)$ must be $C^\M$ regular. In fact, we have:
\begin{prop} \label{HypoM}
	Let $(M, g)$ ba compact connected boundaryless Riemannian manifold that belongs to the class $C^\M$, where $\M$ is a regular sequence with moderate growth. Then every eigenfunction $\psi_\lambda$ of $\Delta_g$ belongs to the class $C^\M$. More precisely, for every $C^\M$ chart $(U, \phi)$ and every compact $K \subset U$, there exist $C$ and $A$ independent of $\lambda$ such that
	\begin{equation} \label{HypoM-good} \sup_{K}|\d^\alpha \psi_\lambda| \leq C e^{\sqrt{\lambda}} A^{|\alpha|} \alpha! \; \M_{|\alpha|} \; \sup_U{| \psi_\lambda}|. \end{equation}
\end{prop}
The proof is very similar to the one of Prop \ref{Hypo}. Namely, we only need to show that the harmonic extension $h$ of $\psi_\lambda$ to $M \times \R$, satisfies $C^\M$-hypoellipticity estimates. However, this follows from the proof of Theorem 4.1 of \cite{Albanese}, from which equivalent estimates can be driven for the Fourier transform of $h$.

\section{The measure of sublevel sets of eigenfunctions} 
In the article \cite{LMa}, in the course of the proof of a propagation of smallness result, the authors find the following lower bounds on the supremum of $|\psi_\lambda|$ over an arbitrary measurable subset $E \subset M$:
\begin{equation} \label{propogation}
\sup_E|\psi_\lambda | \geq \frac{1}{K} \sup_M|\psi_\lambda| \left ( \frac{|E|}{K|M|}\right )^{K \sqrt{\lambda}}.
\end{equation}
Here, $| \cdot |$ refer to the Riemannian volume measure associated to $g$ and the constant $K>1$ is independent of $(\lambda, \psi_\lambda)$ and $E$. 

The estimates \eqref{propogation} resemble the Remez inequalities for polynomials: let $p_n(x)$ be a polynomial of degree $n$ on an interval $I \subset \R$. Then for any measurable subset $E \subset I$, 
\begin{equation*} \label{Remez}
\sup_E|p_n | \geq  \sup_I|p_n| \left ( \frac{|E|}{4|I|}\right )^{n}.
\end{equation*}
The lower bounds \eqref{propogation} also strengthen \emph{the tunneling estimates} by allowing $E$ to be an arbitrary measurable set that can depend on $\lambda$,  instead of an open set independent of $\lambda$, which is assumed in standard tunneling estimates in semiclassical analysis. 

If in \eqref{propogation}, we let $E$ to be the sublevel set $\{ |\psi_\lambda| \leq t \}$, and if we assume that $\psi_\lambda$ is $L^\infty (M)$ normalized, then we obtain (see also Lemma 4.2 and Remark 4.3 of \cite{LMa}):
\begin{equation}\label {sublevel}
| \{ |\psi_\lambda| \leq t \} | \leq K_1 t^{\frac{1}{K\sqrt{\lambda}}}, 
\end{equation}
where $K_1$ and $K$ are independent of $\lambda$ and $t$. 
Interestingly, the bounds \eqref{sublevel} are analogous to P\'olya's inequality on the size of the sublevel sets of polynomails with leading coefficient equal to one:
$$ | \{x \in \R; \, |p_n(x)| \leq t \} | \leq 4 \left ( \frac{t}{2} \right )^{\frac{1}{n}}. $$
\begin{rema}
	The bounds \eqref{sublevel} can also be obtained directly from Lemma 4.2 of \cite{LMa}, which concerns the measure of sublevel sets of harmonic functions in terms of their doubling indices.   
\end{rema}

\section{Proof in the Gervey case}
In this section we prove Theorem \ref{main} which concerns the Gevrey case. Since the arguments for the qusi-analytic, or more generally the Denjoy-Carleman cases, are notationally complicated and less intuitive, we postpone  their presentation to the last section of the article. 
\subsection{Reduction to a local problem}
We cover $M$ with finitely many (independent of $\lambda$)  $s$-Gevrey charts $\{(U_\alpha, \phi_\alpha)\}_{\alpha \in \mathcal F}$ such that each $U_\alpha$ is identified (under $\phi_\alpha$) with the $n$-dimensional Euclidean ball $$B_2^n= \{x \in \R^n; \; |x|<2\},$$ and in addition we require that $\{ \phi_\alpha^{-1}(B^n_1)\}_{\alpha \in \mathcal F}$ covers $M$. To prove our main theorem it suffices to prove that for each $\alpha \in \mathcal F$, there exists $C_\alpha$ such that 
$$\mathcal{H}^{n-1} (Z_{\psi_\lambda} \cap \phi_\alpha^{-1}( B^n_1)) \leq C_\alpha {\lambda}^{\frac{s}{2}}, $$
as this would imply that, with $C=\sum_{\alpha \in \mathcal F}C_\alpha$,
$$ \mathcal{H}^{n-1} (Z_{\psi_\lambda}) \leq \sum_{\alpha \in \mathcal F} \mathcal{H}^{n-1} (Z_{\psi_\lambda} \cap \phi_\alpha^{-1}( B^n_1)) \leq C {\lambda}^{\frac{s}{2}}.$$

Fix a chart $(U, \phi)$ in the covering specified above. From here on, we consider the local eigenfunction $ \tilde{\psi}_\lambda = \psi_\lambda \circ \phi^{-1}$ on the ball $B^n_2$ and estimate the size of its nodal in the ball $B^n_1$. By abuse of notation we will call this local eigenfunction $\psi_\lambda $ instead of $\tilde \psi_\lambda$. 

\subsection{Crofton's formula and the proof of Theorem \ref{main}}

The nodal set $Z_{\phi_\lambda}$ is not necessarily a smooth hypersurface but the $(n-1)$-Haussdorf measure of its singular set is zero.   Thus Crofton's formula can be applied to calculate $\mathcal{H}^{n-1} (Z_{\psi_\lambda} \cap  B^n_1)$ in terms of the number of intersection points of the nodal set with lines going through $x \in B^n_1$ in the direction $\xi \in S^{n-1}$. In fact, we have the following upper bound  (\cite{F}, \cite{Ze2}, \cite{DF1})
\begin{equation}\label{Crofton} \mathcal{H}^{n-1} (Z_{\psi_\lambda} \cap  B^n_1) \leq  C_1 \iint_{B^n_1 \times S^{n-1} } {n}(x, \xi) \, dx d\xi, \end{equation}
where for each $(x, \xi) \in B^n_1 \times S^{n-1}$,  $n(x, \xi)$ is the number of intersection points of the line $x + \R \xi$ with the nodal set $Z_{\psi_\lambda}$, in the ball $B^n_1$. The constant $C_1$ depends only on $n$. Note that $dx d\xi$ is the Liouville measure $d \mu$ on $S^*B^n_1 = B^n_1 \times S^{n-1}$. The function $n(x, \xi)$ can take the value $+\infty$ but only on a set of measure zero. In fact $n \in L^1(S^*B^n_1, d \mu)$. Now for each $m \in \mathbb N$, we denote 
\begin{equation}\label{Am}
A_m = \{ (x, \xi) \in S^*B^n_1; \; (m\sqrt{\lambda})^{s} \leq  n(x, \xi) < ((m+1)\sqrt{\lambda})^{s} \}.
\end{equation}
Because $n(x+t \xi, \xi) = n(x, \xi)$, the measurable subsets $A_m$ are unions of lines.

The following lemma on the size of $A_m$ is crucial in the proof of Theorem \ref{main}.

\begin{lemm} \label{main lemma}There exist $c$ and $C_2$ both positive, dependent only on $(M, g)$ and $s$, such that for all $m \in \mathbb N$ and $\lambda \geq 1$, we have,
	$$ \mu(A_m)= \iint_{A_m} dx d\xi \leq C_2  e^{-c{m}}.$$
\end{lemm}
Before we get into the proof of this lemma let us present how one can apply this lemma and easily obtain a proof for Theorem \ref{main}. 

\begin{proof}[\textbf{Proof of Theorem \ref{main}}] Using \eqref{Crofton} and  the definition of $A_m$, we have 
\begin{align*}
\mathcal{H}^{n-1} (Z_{\psi_\lambda} \cap  B^n_1) & \leq  C_1 \iint_{S^*B^n_1 } n(x, \xi) \, dx d\xi \\
                                                           								& = C_1 \sum_{m=0}^\infty  \iint_{A_m } n(x, \xi) \, dx d\xi \\
                                                           								& \leq C_1 \sum_{m=0}^\infty \mu(A_m) \sup_{(x, \xi) \in A_m} n(x, \xi) \\
                                                           								& \leq C_1 \sum_{m=0}^\infty \mu(A_m) ((m+1) \sqrt{\lambda})^s \\
                                                           								& \leq C_1 C_2 \lambda^{\frac{s}{2}} \sum_{m=0}^\infty  (m+1)^s e^{-cm} \leq C \lambda^{\frac{s}{2}}. 
\end{align*}

\end{proof}
It remains to prove the main lemma.

\section{Proof of estimates on the size of $A_m$}
In this section we prove Lemma \ref{main lemma}. As we shall see, this lemma would follow from another key lemma which we now state but prove later. Throughout this section we assume that $\lambda \geq 1$ and
$$\sup_{B^n_2} |\psi_\lambda| =1. $$
Recall that 
$$A_m = \{ (x, \xi) \in S^*B^n_1; \; (m\sqrt{\lambda})^{s} \leq  n(x, \xi) < ((m+1)\sqrt{\lambda})^{s} \}.$$ 

\begin{lemm} \label{I}
	For each $(x, \xi) \in A_m$, consider the line segment 
	$$I_{x, \xi}:= (x+ \R\xi)\cap B^n_1, $$ equipped with the $1$-dimensional Lebesgue measure $\mathcal L^1$ induced from $\R$. Then, there exist $\ell \in \N^{>0}$ and an open subset $I^* \subset I_{x, \xi}$ with $\mathcal L^1 (I^*) \geq \frac{\mathcal L^1 (I_{x, \xi})}{4^{\ell}}$, such that
	$$ \sup_{I^*} |\psi_\lambda|  \leq C_3 e^{(1-c_0m \ell )\sqrt{\lambda}}.$$
	The constants $c_0$ and $C_3$ are positive and only depend on $(M, g)$ and $s$. 
	\end{lemm}
\begin{rema}
	We shall use $\mathcal L^n$ for the $n$-dimensional Lebesgue measure, however, sometimes we simply use $| \cdot |$. 
\end{rema}

\begin{proof}[\textbf{Proof of Lemma \ref{main lemma}}] We only need the result of Lemma \ref{I}. In fact, the takeaway of that lemma is that for each $(x, \xi) \in A_m$, there exists $\ell \in \N^{>0}$, such that
		\begin{equation} \label{cor} \mathcal L^1 ( I_{x, \xi} )\leq 4^{\ell} \, \mathcal L^1 \left ( I_{x, \xi} \cap \left \{ z \in B^n_1; \; |\psi_\lambda(z) | \leq C_3 e^{(1-c_0m \ell) \sqrt{\lambda}} \right \}\right). \end{equation}
	Before giving the proof let us set some notations. For each $m \in \N$ and $\ell \in \N^{>0}$, we denote the $\mu$-measurable sets
	$$C_{m,\ell}= \{(x, \xi) \in A_m; \quad  \text{\eqref{cor} holds} \}. $$
	By Lemma \ref{I}, 
	$$A_m = \bigcup_{\ell=1}^\infty C_{m , \ell}.$$ 
	We also denote
	\begin{equation*} L_{m, \ell} = \left \{ z \in B^n_1; \; |\psi_\lambda(z) | \leq C_3 e^{(1-c_0m \ell) \sqrt{\lambda}} \right \}, \end{equation*} and 
	$$ B_{m, \ell} = C_{m,\ell} \cap ( L_{m, \ell} \times S^{n-1}).$$
	Note that $B_{m, \ell}$ is a $\mu$-measurable subset of $C_{m, \ell}$. Since by the powerful estimate \eqref{sublevel}, one has
	$$ \mathcal L ^n (L_{m, \ell}) \leq C_4 e^{-\frac{c_0m\ell}{K}},$$
	by letting $c_1= \frac{c_0}{K}$, we get
	$$ \mu(B_{m, \ell}) \leq C_5 e^{-c_1m\ell}.$$
	We claim that
	\begin{equation} \label{claim0} \mu(C_{m, \ell}) \leq 4^{\ell} \mu (B_{m, \ell}). \end{equation}
	If we know this estimate, then 
\begin{align*}
\mu(A_m) \leq \sum_{\ell=1}^\infty\mu(C_{m, \ell}) &  \leq \sum_{\ell=1}^\infty 4^{\ell}\mu(B_{m, \ell}) 
 \leq C_5 \sum_{\ell=1}^\infty 4^{\ell}e^{-c_1m \ell}   \leq    C_5 \sum_{\ell=1}^\infty e^{(2-c_1m) \ell}.                                                                               
\end{align*} 
Therefore, for $m \geq \frac{4}{c_1}$, by the geometric sum we get
$$ \mu(A_m) \leq C_6  e^{2-c_1m} \leq C_6  e^{-\frac{c_1}{2}m}, $$
which provides the estimate of Lemma \ref{main lemma} with $c=\frac{c_1}{2}$. For $m < \frac{4}{c_1}$, Lemma \ref{main lemma} is trivial.

It remains to prove the claim \eqref{claim0}. We shall only use \eqref{cor} and Fubini's Theorem. For simplicity, we drop the sub-indices $(m, \ell)$ from \eqref{claim0} and denote $C= C_{m, \ell}$ and $B=B_{m, \ell}$.  Let $\mathds{1}_{C}(x, \xi)$ be the indicator function of $C$. Since $C$ is $\mu$-measurable, by Fubini's Theorem, for almost every $\xi \in S^{n-1}$ the function $\mathds{1}_{C}(\cdot, \xi)$ is Lebesgue integrable on $B^n_1$ and,
	$$ \mu(C)= \iint_{B^n_1 \times S^{n-1}} \mathds{1}_{C}(x, \xi) d \mu = \int_{S^{n-1}} \left ( \int_{B^n_1}  \mathds{1}_{C}(x, \xi) dx \right ) \, d\xi = \int_{S^{n-1}} \mathcal L^{n} \left ( {(C)_\xi} \right ) \, d\xi.$$
	Here, $(C)_\xi = \{x \in B^n_1; (x, \xi) \in C \}$, which is measurable for almost every $\xi$.  We can write the same identity for $\mu(B)$, hence to prove \eqref{cor}, it suffices to prove that 
	\begin{equation*}
	\mathcal L^{n} ( (C)_\xi )\leq 4^{\ell} \, \mathcal L^{n} ((B)_\xi),  \quad \text{almost every  $\xi \in S^{n-1}$}.
	\end{equation*}
	Note that $(C)_\xi$ is a union of lines parallel to $\xi$ because by the definition \eqref{Am} of $A_m$, whenever $(x, \xi) \in A_m$, one has $ (x + \R \xi ) \cap B^n_1 \subset A_m$ and thus this property would also hold for $C=C_{m, \ell}$, i.e., $(x + \R \xi ) \cap B^n_1 \subset C$ whenever $(x, \xi) \in C$. We also note that if $\mathcal L^{n} ( (C)_\xi )=0$,  there is nothing to prove because then $\mathcal L^{n} ((B)_\xi) =0$ and the above estimate would be trivial.
	
	Now, given $\xi$ with $\mathcal L^{n} ( (C)_\xi ) \neq 0$, we change coordinates in $B^n_1$ by a rotation $x= R(y)$ so that $\xi= R(e_1)$.  Then,
	$$ \mathcal L^{n}  ( (C)_\xi)=\int_{B^n_1}  \mathds{1}_{(C)_\xi}(x) dx = \int_{B^{n}_1}  \mathds{1}_{R^{-1}((C)_\xi)}(y) dy.$$
	Obviously, the set $R^{-1}((C)_\xi)$ is measurable for almost all $\xi$ and is a union of lines parallel to the $y_1$-axis. Using Fubini's Theorem again, we can write the last integral as
	$$\mathcal L^{n}  ( (C)_\xi )= \int_{B^{n}_1}  \mathds{1}_{R^{-1}((C)_\xi)}(y) dy = \int_{B^{n-1}_1} \mathcal L^1 \left ({ (R^{-1}((C)_\xi))_{y'}} \right ) dy', $$
	where $(R^{-1}((C)_\xi))_{y'} = \{y_1 \in [-1, 1]; \; y=(y_1, y') \in R^{-1}((C)_\xi) \}$ is Lebesgue measurable in $[-1, 1]$ for almost every $y' \in B^{n-1}_1$. Again, a similar identity holds for $\mathcal L^{n}  ( (B)_\xi )$. Therefore, in order to prove \eqref{claim0}, we only need to show that 
	$$ \mathcal L^1 \left ({ (R^{-1}((C)_\xi))_{y'}} \right ) \leq 4^\ell \mathcal L^1 \left ({ (R^{-1}((B)_\xi))_{y'}} \right ).$$
	However, this is precisely \eqref{cor}. 
	\end{proof}
In the next section we prove Lemma \ref{I}, which is the technical part of the paper.

\section{Proof of smallness of $\psi_\lambda$ on line segments containing large number of zeros}
We will use the following simple but important lemma.
\begin{lemm}\label{Rolle's}
Let $u$ be a smooth real-valued function on an interval $I \subset \R$. Assume $u$ has at least $p$ many zeros (counting multiplicities)
in $I$. Then,
$$ \sup_{I} | u| \leq \frac{|I|^p}{p!} \sup_{I} |u^{(p)}|,$$
where $|I|$ denotes the length of the interval $I$. 
 \end{lemm}
\begin{proof} A simple proof using  Rolle's theorem can be found in \cite{Hi}[Lemma 4.a]. \end{proof}
\begin{rema}
	This inequality was also proved and used via Hermite interpolation theorem by \cite{IK}[Theorem 2.4] in studying zeros of $1D$ parabolic partial differential equations with Gevrey coefficients.
\end{rema}
The following corollary of the lemma  will be used many times in the course of the proof of Lemma \ref{I}.
\begin{prop}\label{smallness}
	Let $(M, g)$ be a compact Riemannian manifold with no boundary, that belongs to the class $G^s$, $s \geq 1$. Let $\psi_\lambda$ be an eigenunction of $-\Delta_g$ with eigenvalue $\lambda \geq 1$. Assume  $\sup_{B^n_2} |\psi_\lambda| =1$. Consider $(x, \xi) \in S^*B^n_1$ and assume $\psi_\lambda$ has at least $p$ zeros on a subinterval $I$ of  $I_{x, \xi} = (x+ \R\xi) \cap B^n_1$. Then 
	$$ \sup_I |\psi_\lambda| \leq C_7 (A|I|)^p \, p!^{s-1} e^{\sqrt{\lambda}},$$
	where $A$ and $C_7$ depend only on $(M, g)$ and $s$. 
\end{prop}
\begin{rema} \label{DC remark}
This estimate is only useful when $|I|$ is very small in terms of $p$. We also note that when $(M, g)$ belongs to a Denjoy-Carleman class $C^\M$ with $\M$ being a regular weight sequence having moderate growth, in the above estimate we need to replace $p!^{s-1}$ with $\M_{|p|}$.
\end{rema}
\begin{proof}  Observe that $I$  must be of the form $I= x + J\xi$ for some closed interval $J \in \R$. Obviously, $|I|=|J|$. Now we denote
	$$u_\lambda(t) := \psi_\lambda(x + t \xi), \quad t \in J.$$ By assumption, the function $u_\lambda$ has at least $p$ zeros in $J$. Thus by Lemma \ref{Rolle's}, we have
	$$ \sup_{J} | u_\lambda| \leq \frac{|J|^p}{p!} \sup_{J} |u^{(p)}_\lambda|.$$
	Calculating $u_\lambda^{(p)}$ in terms of the partial derivatives of $\psi_\lambda$, we get
	$$ \sup_{J} |u^{(p)}_\lambda| \leq n^p \sup_{|\alpha|= p}\sup_{ B^n_1} | \d^\alpha \psi_\lambda |.$$
	Therefore, by the quantitative hypoellipticity estimates  \eqref{Hypo-good} and the normalization $\| \psi_\lambda\|_{L^\infty(B_2^n)}=1$, we obtain:
$$\sup_{I} |\psi_\lambda| \leq C_7  (A|I| )^p p!^{s-1} e^{\sqrt{\lambda}}.$$
\end{proof} 

\begin{proof}[\textbf{Proof of Lemma \ref{I}}]
Throughout we fix $(x, \xi) \in A_m$. This means that $$ (m\sqrt{\lambda})^{s} \leq  n(x, \xi) < ((m+1)\sqrt{\lambda})^{s}.$$
Recall that, $n(x, \xi) = |I_{x, \xi} \cap Z_{\psi_\lambda} |$, i.e. the number of zeros of $\psi_\lambda$ on the line segment $I_{x, \xi}=(x+ \R\xi)\cap B^n_1$. We may assume during  the proof that $m \geq 32$ and $\lambda \geq 1$, because if $m< 32$, then the estimate in Lemma \ref{smallness} is trivial with the choice of $\ell =1$ and $c_0 \leq \frac{1}{32}$.

Let 
$$n_0=  \lfloor (m\sqrt{\lambda})^{s} \rfloor, \qquad N_0 = \lceil C_0 (m\sqrt{\lambda})^{s-1} \rceil, $$
and 
$$ p_0 = \frac{n_0}{N_0}.$$
The constant $C_0 \geq 1$ will be specified later, independently from $m$, $\lambda$, and $(x, \xi)$. In the first step of our argument, we partition the segment $I_{x, \xi}$ into $N_0$ many equal length subintervals $I_1, \dots, I_{N_0}$, each, of course, of length $\frac{|I_{x, \xi}|}{N_0} < \frac{2}{N_0}$. We then define
$$S_0=\{ 1 \leq j \leq N_0\}, $$
and
\begin{equation} \label{S1} S_1= \{ j \in S_0; \quad |I_j \cap Z_{\psi_\lambda} | \geq \kappa p_0 \},  \end{equation}
where $\kappa \in [\frac14, 1]$ is a universal constant that will be specified later.
Now there are two cases:
\begin{itemize}
	\item[(1a)] \;  $|S_1| > \frac{1}{4}N_0$;
      \item [(1b)] \; $|S_0 \setminus S_1 | > \frac{3}{4} N_4$. 
\end{itemize}
If case (1a) happens, then by Prop \ref{smallness} with $p = \lfloor \kappa p_0\ \rfloor$, we have for each $j \in S_1$:
$$\sup_{I_j} |\psi_\lambda| \leq C_7  (A|I_j| )^p p!^{s-1} e^{\sqrt{\lambda}}.$$
Next, using that $p! \leq p^p$, $\kappa \in [\frac14, 1]$, and $m \sqrt{\lambda} \geq 32$, we get for $j \in S_1$:
\begin{equation}\label{array}
\begin{array}{r@{}l}
\sup_{I_j} |\psi_\lambda| 
& \leq C_7 \left (2A\frac{p^{s-1}}{N_0} \right )^p e^{\sqrt{\lambda}} \\
& \leq  C_7 \left (4A \frac{n_0^{s-1}}{N_0^s} \right )^{\kappa p_0 -1} e^{\sqrt{\lambda}} \\
& \leq C_7  \left (\frac{4A}{C_0^s} \right )^{\frac{m\sqrt{\lambda}}{16C_0}} e^{\sqrt{\lambda}}.
\end{array}
\end{equation}
If we choose $C_0$ so that $\frac{4A}{C_0^s} =\frac{1}{16e}$, then Lemma \ref{I} follows with the choice of $\ell=1$, $I^* = \cup_{j \in S_1} I_j$, and $c_0= \min \{ \frac{1}{16C_0}, \frac{1}{32} \}$. Note that 
$$|I^*| = \sum_{j \in S_1} |I_j| \geq \frac{N_0}{4} \cdot \frac{|I_{x,\xi}|}{N_0} = \frac{|I_{x,\xi}|}{4},$$
which confirms the condition on $I^*$ in the lemma. 

The case (1b), i.e. $|S_0 \setminus S_1 | > \frac{3}{4} N_0$, is more complicated and splits into two subcases:

\begin{itemize}
	\item[($1b'$)] \;  $\sum_{j \in S_0 \setminus S_1} |I_j \cap Z_{\psi_\lambda}| \geq \frac{n_0}{2}$;
	\item [($1b''$)] \; $\sum_{j \in S_0 \setminus S_1} |I_j \cap Z_{\psi_\lambda}| < \frac{n_0}{2}$. 
\end{itemize}
Assume ($1b'$) holds. We claim that with the choice of $\kappa = \frac{39}{60}$ in \eqref{S1}, we have
\begin{equation}\label{claim} \left | \{ j \in  S_0 \setminus S_1; \quad |I_j \cap Z_{\psi_\lambda}|  \geq \frac{p_0}{4}\} \right | \geq \frac{3}{8} N_0.
\end{equation}
To prove this, first note that
$$ \sum_{j \in S_0 \setminus S_1}  \kappa p_0 - |I_j \cap Z_{\psi_\lambda}|  =\kappa p_0 |S_0 \setminus S_1| - \sum_{j \in S_0 \setminus S_1}|I_j \cap Z_{\psi_\lambda}| \leq \kappa p_0 N_0 -  \frac{n_0}{2} = \left (\kappa - \frac12 \right )n_0.$$
Thus,
\begin{equation}\label{average} \frac{1}{|S_0 \setminus S_1|} \sum_{j \in S_0 \setminus S_1}  \kappa p_0 - |I_j \cap Z_{\psi_\lambda}| \leq  \left (\kappa - \frac12 \right )\frac{n_0}{|S_0 \setminus S_1|} \leq \frac{4}{3} \left (\kappa - \frac12 \right )p_0. \end{equation}
Now we implement the following simple lemma for the above average estimate. 
\begin{lemm}
	Let $\{z_j\}_{j=1}^d$ be a set of non-negative real numbers such that
	$$ \frac{z_1 + \dots +z_d}{d} \leq T.$$
	Then the subset $S= \{ j; \; z_j \leq 2T \}$ has at least half density in $\{1, 2, \dots, d \}$,  i.e. $\frac{|S|}{d} \geq  \frac12$. 
\end{lemm}
\begin{proof}The proof is elementary. Let $S^c$ be the complement of $S$. Then
	$$  \frac{2T |S^c|}{d}  <  \frac{\sum_{j \in S^c} z_j}{d} \leq  \frac{\sum_{j=1}^d z_j}{d} \leq T,$$
which implies that $\frac{|S^c|}{d} < \frac12$.  \end{proof}
Applying this lemma to the set $\left \{\kappa p_0 - |I_j \cap Z_{\psi_\lambda}|; \quad j \in S_0 \setminus S_1 \right \}$, we obtain:
$$ \left|j \in S_0 \setminus S_1; \quad \kappa p_0 - |I_j \cap Z_{\psi_\lambda}|  \leq \frac{8}{3} \left (\kappa - \frac12 \right )p_0 \right|  \geq \frac{|S_0 \setminus S_1|}{2} \geq   \frac{3}{8} N_0.$$
However, 
$$\kappa p_0 - |I_j \cap Z_{\psi_\lambda}|  \leq \frac{8}{3} \left (\kappa - \frac12 \right )p_0 \implies |I_j \cap Z_{\psi_\lambda}| \geq  \left (\kappa - \frac{8}{3} \left (\kappa - \frac12 \right ) \right ) p_0.$$
The claim \eqref{claim} follows since for $\kappa = \frac{39}{60}$, we have $\kappa - \frac{8}{3} \left (\kappa - \frac12 \right ) = \frac14$. 
Given \eqref{claim}, we use Prop \ref{smallness} again, with $p = \lfloor \frac{p_0}{4} \rfloor$, and follow the same array of estimates as in \eqref{array}. In fact we get exactly the same estimate with the same $C_0$, $c_0$, and $\ell$. Even the length of the obtained union of subintervals is bounded below by the same quantity: 
$$|I| \geq  \frac{3N_0}{8} \cdot \frac{|I_{x,\xi}|}{N_0} = \frac{3|I_{x,\xi}|}{8} > \frac{I_{x, \xi}}{4}.$$

Next, we need to deal with the case ($1b''$), which turns out to be more involved. 

So let us assume $|S_0 \setminus S_1 | > \frac{3}{4} N_0$ and $\sum_{j \in S_0 \setminus S_1} |I_j \cap Z_{\psi_\lambda}| < \frac{n_0}{2}$. Equivalently, we have:
 $$|S_1 | \leq  \frac{1}{4} N_0 \quad  \text{and} \quad \sum_{j \in S_1} |I_j \cap Z_{\psi_\lambda}| \geq  \frac{n_0}{2}.$$ 
 In the second step of our argument, we let 
 $$ N_1 = |S_1|, \quad n_1=  \frac{n_0}{2}  \quad \text{and} \quad p_1=  \frac{n_1}{N_1}.$$
 We  also denote
$$S_2 =\{ j  \in S_1;\ \quad |I_j \cap Z_{\psi_\lambda}| \geq \kappa p_1  \}. $$
We repeat the same procedure we previously applied to $S_1$, to $S_2$. So we get cases 
($2a$), ($2b'$), and ($2b''$) and everywhere we replace $p_0$ with $p_1$, $n_0$ with $n_1$, and $N_0$ with $N_1$. In fact, by continuing this procedure, we obtain for each $k \in \N^{>0}$,  subsets $S_k$ defined inductively by
$$S_k= \{ j \in S_{k-1}; \quad |I_j \cap Z_{\psi_\lambda}| \geq  \kappa p_{k-1}\}.$$
We also have for each $k$:
 $$ N_k = |S_k|, \quad n_k= \frac{n_0}{2^k}  \quad \text{and} \quad p_k=  \frac{n_k}{N_k}. $$
 Furthermore, we have
 $$|S_k | \leq  \frac{1}{4} N_{k-1} \leq \frac{1}{4^k} N_{0} \quad  \text{and} \quad \sum_{j \in S_k} |I_j \cap Z_{\psi_\lambda}| \geq  \frac{n_{k-1}}{2} = \frac{n_0}{2^k}.$$ 
 This procedure is not infinite, because as soon as  $\frac{1}{4^k} N_{0} <1$ we get $S_k = \emptyset$, and the procedure stops. This corresponds to $k = \lceil \frac{\log N_0}{\log 4}  \rceil+1$.  To prove Lemma \ref{smallness}, we need to show that for each of the intermediate cases ($ka$) and ($kb'$), the lemma follows. So let us fix $k \in \N^{>0}$. Cases ($ka$) and ($kb'$) are:
 \begin{itemize}
 	\item [($ka$)] \;  $|S_k| > \frac{1}{4}N_{k-1}$
 	\item [($kb'$)] \;  $|S_{k-1} \setminus S_k | > \frac{3}{4} N_{k-1}$ and $\sum_{j \in S_{k-1} \setminus S_k} |I_j \cap Z_{\psi_\lambda}| < \frac{n_{k-1}}{2}$.
 	\end{itemize}
 Suppose case ($ka$) occurs. By the definition of $S_k$, for each $j \in S_k$, we have at least $\kappa p_{k-1}$ zeros of $\psi_\lambda$ on $I_j$. Now we partition each interval $I_j$, $j \in S_k$, into $\lfloor \frac{p_{k-1}}{p_0} \rfloor$ many equal length intervals. Hence there must exist $J_j \subset I_j$ of length $\frac{|I_j|}{\lfloor \frac{p_{k-1}}{p_0} \rfloor}$ on which $\psi_\lambda$ has at least 
 $$\frac{\kappa p_{k-1}}{\lfloor \frac{p_{k-1}}{p_0} \rfloor} \geq \kappa p_0$$
 many zeros.  
 So using Prop \ref{smallness}, with $I= J_j$ and $p = \lfloor \kappa p_{0} \rfloor$, we obtain
 \begin{equation}\label{array2}
 \begin{array}{r@{}l}
 \sup_{J_j} |\psi_\lambda| 
 &\leq C_7  (A|J_j| )^p p!^{s-1} e^{\sqrt{\lambda}} \\
 & \leq C_7 \left (4A\frac{p^{s-1}}{N_0}  \cdot \frac{p_{0}}{p_{k-1}} \right )^p e^{\sqrt{\lambda}} \\
 & \leq  C_7 \left (4 A \frac{n_0^{s-1}}{N_0^s}  \cdot \frac{p_{0}}{p_{k-1}} \right )^{\kappa p_0 -1} e^{\sqrt{\lambda}} \\
 & \leq C_7  \left (\frac{4A}{C_0^s}  \cdot \frac{p_{0}}{p_{k-1}} \right )^{\frac{m\sqrt{\lambda}}{16C_0}} e^{\sqrt{\lambda}}.
 \end{array}
 \end{equation}
 Since in \eqref{array}, we have already chosen $C_0$ and  $c_0$ to satisfy 
 $\frac{4A}{C_0^s} = \frac{1}{16e}$ and $c_0 = \min\{\frac{1}{16C_0}, \frac{1}{32}\}$, we obtain 
 \begin{equation}\label{p/p}\sup_{J_j} |\psi_\lambda| \leq C_7  \left (\frac{p_{0}}{16ep_{k-1}} \right )^{c_0m \sqrt{\lambda}} e^{\sqrt{\lambda}}. \end{equation} The total interval on which we have this estimate is 
 $I^* = \cup_{j \in S_k } J_j$, which has length
 \begin{align*}
 | I^*| = | S_k| |J_j| & \geq  \frac{1}{4} N_{k-1} \cdot \frac{|I_j|}{\lfloor \frac{p_{k-1}}{p_0} \rfloor} =  \frac{N_{k-1}}{4N_0} \cdot \frac{|I_{x, \xi}|}{\lfloor \frac{p_{k-1}}{p_0} \rfloor} \\
                                   & \geq \frac{n_{k-1}}{8n_0}  \cdot \left ( \frac{p_{0}}{p_{k-1}} \right)^2 {|I_{x, \xi}|} \\
                                   & = \frac{1}{2^{k+2}}  \cdot \left ( \frac{p_{0}}{p_{k-1}} \right)^2 {|I_{x, \xi}|}.
  \end{align*}
  Let $\ell \in N^{>0}$ be the unique integer that satisfies 
  \begin{equation} \label{ell}\frac{1}{4^\ell} \leq \frac{1}{2^{k+2}}  \cdot \left ( \frac{p_{0}}{p_{k-1}} \right)^2 < \frac{1}{4^{\ell-1}}. \end{equation}
By this choice of $\ell$, $|I^*|$ certainly satisfies the required lower bound $\frac{|I_{x, \xi}|}{4^\ell}$. Applying this to \eqref{p/p}, we get
\begin{equation}\label{sup}\sup_{J_j} |\psi_\lambda| \leq C_7  \left (2^{k} 4^{-\ell} \right )^{\frac{c_0}{2}m \sqrt{\lambda}} e^{\sqrt{\lambda}}. \end{equation}
On the other hand, by the definition of $p_{k-1}$, we have 
$$ \frac{p_{k-1}}{p_0} = \frac{ \frac{n_{k-1}}{N_{k-1}}}{p_0} \geq \frac{ \frac{n_0/2^{k-1}}{N_0/4^{k-1}}}{2p_0} \geq {2^{k-1}}.$$
Using this and \eqref{ell}, we must have that $4^\ell \geq 8^k$, or equivalently $k \leq \frac23 \ell$, therefore \eqref{sup} becomes
\begin{equation}\sup_{J_j} |\psi_\lambda| \leq C_7  \left ( 2^{-4/3 \ell} \right )^{\frac{c_0}{2}m \sqrt{\lambda}} e^{\sqrt{\lambda}} =  C_7 e ^ {(1 - \frac{2}{3} (\log 2) {c_0}m \ell) \sqrt{\lambda}} \leq C_7 e ^ {(1 - {c_0}m \ell) \sqrt{\lambda}} . \end{equation}
This proves the lemma in the case ($ka$). Finally assume case ($kb'$) occurs. Then as in the case ($1b'$), we claim and prove that for $\kappa= \frac{39}{60}$, 
\begin{equation}\label{claim2} \left | \{ j \in  S_{k-1} \setminus S_k; \quad |I_j \cap Z_{\psi_\lambda}|  \geq \frac{p_{k-1}}{4}\} \right | \geq \frac{3}{8} N_{k-1}.
\end{equation}
Up to a change of indices, the proof of \eqref{claim2} is identical to the proof of claim \eqref{claim}. We then follow as in the case ($ka$), meaning, for each $j \in S_{k-1} \setminus S_k$ that satisfies the lower bound $|I_j \cap Z_{\psi_\lambda}|  \geq \frac{p_{k-1}}{4}$, we find a subinterval $J_j\subset I_j$ with length $\frac{|I_j|}{\lfloor \frac{p_{k-1}}{p_0} \rfloor}$ on which $\psi_\lambda$ has at least 
$\lfloor \frac{p_0}{4} \rfloor$ many zeros. We then argue exactly as in the case ($ka$) presented above. 
\end{proof}

\section{Proof in the case of Denjoy-Carleman Riemannian manifolds}
In this section, we sketch the proof of Theorem \ref{general}. A main tool is $C^\M$ hypoellipticity estimates for eigenfunctions which we already presented in \eqref{HypoM-good}. The proof of the upper bounds on the size of nodal sets follows the same arguments as in the Gevrey case but with some modifications that we now present. Assume that $\M_k$ is a regular sequence with moderate growth and let
$$J_k = (k!\M_k)^{1/k}.$$
 It is known that $J_k$ is an increasing sequence. 
Then we change the definition of the set $A_m$ in \eqref{Am}, to
\begin{equation*}
A_m = \left \{ (x, \xi) \in S^*B^n_1; \; J_{m \lfloor \sqrt{\lambda} \rfloor } \leq  n(x, \xi) < J_ {(m+1) \lfloor \sqrt{\lambda} \rfloor} \right \}.
\end{equation*}
As we will see, we still get the same estimates $\mu(A_m) \leq C e^{-cm}$ as in Lemma \ref{main lemma} on the size of $A_m$. Following the same steps as we gave for the proof of Theorem \ref{main}, we get 
\begin{align} \label{DC nodal}
\mathcal{H}^{n-1} (Z_{\psi_\lambda} \cap  B^n_1) 
& \leq C  \sum_{m=0}^\infty e^{-cm} J_{(m+1) \lfloor \sqrt{\lambda} \rfloor }.
\end{align}
To obtain bounds on the last sum, we utilize the following lemma:
\begin{lemm} \label{moderate}
	Let $\M_k$ be a regular sequence with moderate growth. Then there exists a constant $\gamma>0$ such that for all $m, j \in \N^{>0}$, we have
	$$ \left ( \M_{mj} \right ) ^{1/mj} \leq e^{2\gamma} \left (\M_j \right ) ^{1/j}  m^\gamma. $$  
\end{lemm}
\begin{proof} Let $A =  \sup_{j, k \in \N^{>0}} \left( \frac{\M_{j+k}}{\M_j \M_k} \right )^{1/(j+k)}$. By the definition of moderate growth, we have $A< \infty$. Then using the convexity of $\M$ and the moderate growth assumption, one can show that (see \cite{Th}[equations (6) and (9)]): 
$$ \M_{j+1} \leq A^2 \M_j^{1+ \frac{1}{j}}.$$
Using this inequality repeatedly, we obtain for all $j, k \in \N^{>0}$ that
$$ (\M_{j + k}^{})^{1/(j+k)} \leq \M_j^{1/j} A^{2+ 2\sum_{i=0}^{k-1} \frac{1}{j+i}} \leq \M_j^{1/j} A^{4+ 2\log \left ( \frac{j+k}{j}\right ) }.$$
If we put $k=(m-1)j$ into this inequality, then the lemma follows with $\gamma= 2 \log A$.
\end{proof}
\begin{proof}[\textbf{Proof of Theorem \ref{general}}] Let us continue on \eqref{DC nodal}. The proof is straightforward because by the above lemma and the definition of $J_k$:
\begin{align*}
\mathcal{H}^{n-1} (Z_{\psi_\lambda} \cap  B^n_1) 
& \leq C  \sum_{m=0}^\infty e^{-cm} J_{(m+1) \lfloor \sqrt{\lambda} \rfloor } \\
&\leq C_1  \sum_{m=0}^\infty e^{-cm} (m+1)^{1+\gamma} \sqrt{\lambda} \left( \M_{ \lfloor \sqrt{\lambda} \rfloor } \right )^{1/\lfloor \sqrt{\lambda} \rfloor } \\
& \leq  C_2 \sqrt{\lambda} \left ( \M_{ \lfloor \sqrt{\lambda} \rfloor } \right ) ^{1/\lfloor \sqrt{\lambda} \rfloor }. 
\end{align*}
\end{proof}

To prove the exponential upper bound $\mu(A_m) \leq Ce^{-cm}$, we only need to know that Lemma \ref{I} holds in the Denjoy-Carleman setting. In fact there will be no change in the statement of Lemma \ref{I}. As we pointed out in Remark \ref{DC remark}, in the Denjoy-Carleman case there is a change in the estimates of Prop \ref{smallness}: 
$$ \sup_I |\psi_\lambda| \leq C_7 (A|I|)^p \, \M_{|p|} e^{\sqrt{\lambda}}.$$
However, this will not change the result of Lemma \ref{I} as the proof reveals. 
The only changes required are in the initial quantities used in the first step of the proof:
$$n_0= \lfloor J_{m \lfloor \sqrt{\lambda} \rfloor } \rfloor, \quad N_0 = \left \lfloor \frac{n_0}{m \sqrt{\lambda}} \right \rfloor, \quad p_0=\frac{n_0}{N_0}.$$
Correspondingly, these will change the quantities $n_k$, $N_k$, and $p_k$. The proof is otherwise identical to the proof in the Gevrey case.

 \subsection*{Acknowledgement} The author would like to thank Steve Zelditch for their comments on the first draft of this article and for bringing to attention the works of \cite{Hi} and \cite{NSV}. The research of the author is supported by the Simons Collaborations Grants for Mathematicians 638398.

\end{document}